\documentclass{amsart}

\usepackage{latexsym, amssymb, amsmath, epsfig, amscd, amsthm}
\usepackage{amsfonts}
\usepackage{mathrsfs}
 \usepackage{hyperref}
\usepackage{cite}

\allowdisplaybreaks

\numberwithin{equation}{section}

\newcommand{\NN}{\mathbb{N}}

\newcommand{\RR}{\mathbb{R}}
\newcommand{\TT}{\mathbb{T}}
\newcommand{\ZZ}{\mathbb{Z}}

\def\cF{{\mathcal F}}

\def\cL{{\mathcal L}}

\def\cS{{\mathcal S}}

\newcommand{\findif}[1]{\mathop{\triangle^{#1}}}

\DeclareMathOperator{\supp}{\text{supp}}

\newcommand{\dbar}{d\mkern-6mu\mathchar'26\mkern-2mu} 
\newcommand{\bangle}[1]{\left\langle #1 \right \rangle}

\newtheorem{prop}{Proposition}[section]
\newtheorem{theo}[prop]{Theorem}
\newtheorem{lemm}[prop]{Lemma}

\newtheorem{defi}[prop]{Definition}

\title[Infinite Matrix Reps. of Isotropic Pseudodifferential Operators] {Infinite Matrix Representations of Isotropic Pseudodifferential Operators}
\author{Otis Chodosh}
\date{\today}

\begin{document}

\begin{abstract} We characterize isotropic pseudodifferential operators (elements of the Shubin calculus) by their action on Hermite functions. We show that a continuous linear operator $A:\cS(\RR)\to\cS'(\RR)$ is an isotropic pseudodifferential operator of order $r$ if and only if its ``matrix'' $(K^{(A)})_{m,n} := \bangle{A \phi_n,\phi_m}_{L^2(\RR)}$ is rapidly decreasing away from the diagonal $\{m=n\}$, order $r/2$ in $m+n$, where applying the discrete difference operator along the diagonal decreases the order by one. Here $\phi_{m}$ is the $m$-th Hermite function. As an application, we give an isotropic version of the Beals commutator characterization of pseudodifferential operators, showing that if we define $H = -\partial_{x}^{2} + x^{2}$ to be the harmonic oscillator and $Z$ the map extended linearly from $Z(\phi_{k}) = \phi_{k-1}$, then a continuous linear operator $A:\cS(\RR)\to\cS'(\RR)$ is an isotropic pseudodifferential operator of order $r$ if and only if commuting $A$ $\alpha$ times with $H$ and $\beta$ times with $Z$ results in an bounded linear operator $H^{s+r-2\beta}_{\text{iso}}(\RR) \to H^{s}_{\text{iso}}(\RR)$, for all $s\in \RR$ and $\alpha,\beta \in \NN_{0}$. 
\end{abstract}
\maketitle
\tableofcontents
 \section{Introduction}

In this paper we examine the action of isotropic pseudodifferential operators on the Hermite functions. We give a necessary and sufficient condition for an operator to be an isotropic pseudodifferential operator based on the form of action. This result is a consolidation of results originating in the author's honors thesis, \cite{chod:thesis}, which contains, in addition to the material given here, a proof of a similar theorem for pseudodifferential operators on a torus (which seems to be well known, although a proof could not be located in the literature) as well as more detailed background about Hermite functions and isotropic pseudodifferential operators. In addition to the above result, we give an application of our discretization, proving a commutator characterization of isotropic pseudodifferential operators in the sense of Beals.   

 In order to state our main theorem we first give the following definition:
\begin{defi}\label{defi:SMr}
If we define the discrete difference operator $\triangle$  on a function $K: \NN_0\times\NN_0 \to\RR$ by  \[(\triangle K)(m,n) = K(m+1,n+1) - K(m,n)\] (writing $\triangle^\alpha$ to signify applying the difference operator $\alpha$ times), then we will say that a function $K: \NN_0\times\NN_0 \to\RR$ is a \emph{symbol matrix of order $r$} if for all $\alpha, N \in \NN_0$, there is $C_{\alpha,N} >0$ such that 
\[
| (\triangle^\alpha K)(m,n) | \leq C_{\alpha,N} (1+m + n)^{r -\alpha} (1+|m-n|)^{-N}.
\] We will denote the set of symbol matrices of order $r$, $SM^r(\NN_0)$. 
\end{defi}

With this definition, we can state our main result:
\begin{theo}\label{theo-intro-main}\label{theo-intro-main}
A linear operator, $A:\cS(\RR) \to \cS'(\RR)$ is an isotropic pseudodifferential operator of order $r$ (as defined in Section \ref{sec-Prelim}) if and only if the ``matrix of $A$'' \[K^{(A)}:\NN_0\times \NN_0 \to \RR\] defined by $(m,n)\mapsto \langle A\phi_n, \phi_m\rangle_{L^2(\RR)}$ is an order $r/2$ symbol matrix (where $\phi_n$ is the $n$-th Hermite function, as defined in Section \ref{sec-Prelim}). 
\end{theo}

Perhaps the simplest example of this theorem is when $A$ is a power of the harmonic oscillator. It can be shown that for $s\in \RR$, $(1+H)^{s/2} \in \Psi_\text{iso}^{s}(\RR)$ is an elliptic operator, where $H$ is the harmonic oscillator, $H = -\frac{\partial^2}{\partial x^2} + x^2$ (see, for example, \cite{Shubin:PDO}, Theorems  II.10.1 and II.11.2 concerning powers of elliptic operators). Furthermore, because $(1+H)\phi_n = (2+2n)\phi_n$, it can be shown that $(1+H)^{s/2} \phi_n = (2+2n)^{s/2} \phi_n$. Thus
\[
(K^{(1+H)^{s/2}})_{m,n} = \bangle{ (1+H)^{s/2} \phi_n , \phi_m}_{L^2} = (2+2n)^{s/2} \delta_{m,n}.
\]
Here we have written $(K^{(1+H)^{s/2}})_{m,n}$ when we really mean $K^{(1+H)^{s/2}}(m,n)$, a shorthand we will use frequently below. Notice that this is a symbol matrix of order $s/2$: because $(K^{(1+H)^{s/2}})_{m,n} =0$ for $m\not = n$, it is enough to show that 
\[
|(\triangle^\alpha K^{(1+H)^{s/2}})_{n,n} | \leq C_{\alpha} (1+n)^{s/2-\alpha}.
\]  
This follows from the fact that $\partial_x^k(1+x)^s \lesssim (1+x)^{s-k}$ and the mean value theorem. Thus, powers of the harmonic oscillator, which are some of the simplest examples of isotropic pseudodifferential operators satisfy the statement of the theorem. 

Theorem \ref{theo-intro-main} is similar to the situation for pseudodifferential operators on the $d$-dimensional torus $\TT^d$. In this case, it is well known that for an operator $A:C^\infty(\TT^d) \to C^\infty(\TT^d)$, $A$ is an order $r$ pseudodifferential operator on $\TT^d$ if and only if $K^{(A)}(m,n):=\bangle{ A(e^{in\cdot x}), e^{im\cdot x}}){L^2}$ is a symbol matrix of order $r$ (where we extend the definition of symbol matrix to a map $K:\ZZ^d\times \ZZ^d\to \RR$ in the obvious manner, letting the difference operator act in any of the $d$-directions). For the dimension $d=1$, this is very similar to the isotropic case discussed in this paper. However, for higher dimensions, there is not such a simple characterization of the matrices of $\Psi_\text{iso}^r(\RR^d)$. We discuss this further in the end of Section \ref{main-theo:sec}. A proof of the theorem on the torus can be found in \cite{chod:thesis}.

Related forms of discretization of pseudodifferential operators have been studied by various authors. Ruzhansky and Turunen in \cite{Ruzhansky:torus,Ruzhansky:QuantTorus} have considered Fourier series discretization of toroidal pseudodifferential operators	. In their work, they consider ``quantized'' symbols of the form $a(x,\xi) \in C^{\infty}(\TT^{d}\times \ZZ^{d})$ with 
\[ |\partial_{x}^{\alpha} \findif{\beta}  a(x,\xi)  | \leq C_{\alpha,\beta}  (1+|\xi|)^{m-\beta}\]
where $\findif{\beta}$ is the finite difference operator applied $\beta$ times in the $\xi$ variable. These symbols then define a toroidal pseudodifferential operator $A_{a} \in \Psi^{m}(\TT^{n})$ via the action on test functions $\phi \in C^{\infty}(\TT^{d})$ given by 
\[ A_{a}\phi(x) = \sum_{\xi \in\ZZ^{n}} \int_{\TT^{n}} e^{i(x-y)\cdot\xi} a(x,\xi) \phi(y) \dbar y . \] They show that any pseudodifferential operator on the torus has such a symbol, and describe how the quantized symbols relate to the usual notion. They go on to discuss Fourier intergral operators on the torus and applications to hyperbolic equations. Their quantization is somewhat different than the one described in this paper, because they are quantizing only in the $\xi$ variable, and not in the $x$ variable as well. In addition, we remark that in \cite{RuzhanskyTurunen:PSIDOsymmetries}, they have extended this quantization to pseudodifferential operators over general compact Lie groups, and it is interesting to note that they have also observed a different form of off diagonal decay of the symbols in the setting of e.g.\ matrix valued symbols for operators in $\Psi^{m}(\text{SU}(2))$ (cf.\ \cite{RuzhanskyTurunen:PSIDOsymmetries} Chapter 12). 

Finally, our characterization of isotropic pseudodifferential operators allows us to give the following commutator characterization in the sense of Beals \cite{Beals:CommCharPSIDO}. We define the operator $Z:\cS(\RR)\to \cS(\RR)$ by $\phi_{k}\mapsto \phi_{k-1}$ (extending linearly), and recall that the harmonic oscillator is $H = -\frac{\partial^{2}}{\partial x^{2}} + x^{2}$, and from this, we define the commutation operators $\tilde Z(A) := [A,Z]$ and $\tilde H(A) : =[A,H]$, allowing us to state
\begin{theo}\label{theo:iso-beals}
An operator $A :\cS(\RR) \to \cS'(\RR)$ has $A \in \Psi^{r}_{\text{iso}}(\RR)$ if and only if for all $\alpha,\beta \in \NN_{0}$ and $s\in \RR$
\begin{equation}\label{eq:iso-comm-cond} \tilde H^{(\alpha)}(\tilde Z^{(\beta)}(A)) \in \cL(H^{r+s-2\beta}_{\text{iso}}(\RR) \to H^{s}_{\text{iso}}(\RR)). \end{equation}
\end{theo}
This theorem gives an isotropic version of the classical theorem of Beals, which says that if commutators of a linear operator with $x_{j}$ and $\partial_{j}$ have appropriate mapping properties, then it is a (standard) pseudodifferential operator. It is important to note that the above operator $Z$ is an isotropic pseudodifferential operator, but not a differential operator.

In the following, we fix a choice of Fourier transform, defining
\[ \cF f(\xi) = \hat f(\xi) = \int_{\RR^d} e^{-i\xi x} f(x) \ dx. \] 
With this convention, the inverse Fourier transform is \[\cF^{-1} f(x) = \int_{\RR^d} e^{ix \xi} f(\xi)\  \dbar\xi,\] 
where $\dbar \xi := (2\pi)^{-d} d\xi$. 

Additionally, we will also use the standard notation $f(x) \lesssim g(x)$ to denote that there is some $C> 0$ such that $f(x) \leq C g(x)$ for all $x$. When both sides of the equation depend on multiple variables it should be clear from context which ones $C$ depends on. We also write $\cS(\RR)$ for Schwartz functions ($f :\RR\to\RR$ such that for all $\alpha,\beta \in \NN_{0} = \{0,1,2,\dots\}$, we have that $\| x^{\alpha}\partial_{x}^{\beta} f \|_{L^{2}(\RR)} < \infty$). We further let $\cS'(\RR)$ be the space of tempered distributions, the topological dual to the Schwartz space $\cS(\RR)$.

We have compiled various elementary results about symbol matrices in Appendix \ref{appA}, most of which are used in the proof of our main result, Theorem \ref{theo-intro-main}. Appendixes \ref{appB} and \ref{appC} contain proofs of Lemmas which are needed in the proof of the main theorem, but we felt were somewhat extraneous to the exposition, as they both are somewhat messy calculations. Appendix \ref{appC} could potentially be of independent interest, in it we prove that for $A \in \Psi_\text{iso}^{\epsilon_0} (\RR)$ (for $\epsilon_0<1/2$), if $A$ is bounded $L^2 \to L^2$, then the symbol of $A$ is bounded. As such, it is a partial converse to the $L^2$ boundedness of order $0$ pseudodifferential operators. 

\section{Acknowledgements} 

Much of this paper represents a part of my undergraduate honors thesis. I am very grateful to my advisor, Andr\'as Vasy, for teaching me about pseudodifferential operators and microlocal analysis, and for suggesting and then greatly assisting me with this work. I would like to thank Richard Melrose, Michael Ruzhansky, and Michael Taylor for their helpful comments and suggestions. I would also like to thank the referees for their careful reading and comments. Part of this work was completed while supported by a Stanford University VPUE research grant. 

\section{Isotropic Pseudodifferential Operators and Hermite Functions} \label{sec-Prelim}
  Isotropic pseudodifferential operators, also known as the ``Shubin class'' in the literature, are defined  to be pseudodifferential operators with symbols $a(x,\xi) \in C^\infty(\RR^2)$ such that for all $\alpha,\beta \in \NN_0$, there is a $C_{\alpha,\beta} >0$ so that for $x,\xi\in \RR$
 \[ |\partial_x^\alpha\partial_\xi^\beta a(x,\xi)| \leq C_{\alpha,\beta} (1+|x|+ |\xi|)^{r-\alpha-\beta}. \] Here we follow Melrose's development of the class (c.f.\ \cite{Melrose:MAnotes}). Note that these differ from ``regular'' symbols in two ways, namely that we require decay in both $x$ and $\xi$ (regular symbols only require decay in $\xi)$ and we also require $x$ derivatives to improve the decay (for regular symbols the $x$ derivatives are only required to not make the decay worse). We denote the above class of symbols by $S^r_\text{iso}(\mathbb{R})$, and class of operators with these symbols $\Psi_\text{iso}^r(\RR)$, which we call isotropic pseudodifferential operators of order $r$. 
 
 We recall that given an (isotropic) symbol $a(x,\xi)$, we obtain an (isotropic) pseudodifferential operator $A$ by setting
 \[ A u (x) : = \int_{\RR\times \RR} e^{i(x-y)\cdot \xi} a(x,\xi) u(y) dy\dbar \xi,\]
 for Schwartz functions $u$. In, e.g.\ \cite{Melrose:MAnotes}, it is shown that this is a well defined continuous operator on Schwartz functions. Additionally, it can be shown that the composition of two isotropic pseudodifferential operators is again an isotropic pseudodifferential operator, and that the adjoint of an isotropic pseudo-differential operator is again an isotropic pseudo-differential operator of the same order (see, e.g. Melrose \cite{Melrose:MAnotes} Proposition 4.1 and Theorem 4.1).  Furthermore, it can be shown that elliptic elements of $\Psi^r_\text{iso}(\RR^d)$ have two sided parametrices in $\Psi^{-r}_\text{iso}(\RR^d)$. Additionally, because $\Psi_\text{iso}^0(\RR^d)\subset \Psi^0(\RR^d)$, we have $L^2$ boundedness of order 0 isotropic pseudo-differential operators. 

We will need the notion of  an elliptic isotropic pseudodifferential operator of order $r$, which is is an isotropic pseudodifferential operator whose symbol obeys
\[ |a(x,\xi)| \geq C (1+|x|+|\xi|)^{r}\]
for $x,\xi$ outside of some compact set. We further remark that, as in the case of standard symbols, the full symbol is not uniquely determined by the operator, but its image in $S^{r}_{\text{iso}}/S^{r-1}_{\text{iso}}$ is uniquely determined, and we call this image the principal symbol of the operator $A$. 

One property of isotropic calculus that is different from the standard calculus is closure under the Fourier transform in the following sense: defining $\hat A$ by the formula $\widehat { \hat A u} = A \hat u $, it can be shown that $\hat A$ is a isotropic pseudodifferential operator of the same order with principal symbol $\hat a (x,\xi) : = a(\xi,- x)$. 

Finally we discuss the Hermite functions on $\RR^{d}$. If we define the $j$-th creation operator $C^\dagger_{j} = \frac{1}{\sqrt{2}} (-\frac{\partial}{\partial x_{j}} + x_{j})$ , and let $\phi_0 = \frac{1}{\pi^{1/4}} e^{-|x|^2/2} $, then to define the higher Hermite functions (depending on a multi-index $\alpha$), we inductively let $\phi_{\alpha+e_{j}} = \frac{1}{\sqrt{\alpha_{j}+1}} C_{j}^\dagger \phi_\alpha$. These are an orthonormal basis for $L^2(\RR^{d})$ and also span a dense subset of $\cS(\RR^{d})$. Note that $H\phi_\alpha = (1+2|\alpha|) \phi_\alpha$ (where $H = -\sum_{j=1}^{d}\frac{\partial^{2}}{\partial x_{j}^{2}} + |x|^{2}$ is the harmonic oscillator), and $\mathcal{F} \phi_k(\xi) = (-i)^{\alpha} (2\pi)^{d/2} \phi_k(\xi)$, so in particular the Hermite functions are eigenfunctions of both the harmonic oscillator and the Fourier transform. More properties and proofs of these claims can be found in many places, e.g.\ Chapter 6 of \cite{Taylor:PDE2}. Alternatively, the results that we will make use of are described in detail in Chapter 4 of \cite{chod:thesis}. For later use, we record the following (c.f.\ \cite[Lemma 4.8]{chod:thesis})

\begin{lemm}\label{lem:hmt-fct-der-mult}For a multi-index $n \in \ZZ^d_0$ with $n_k>0$ for $1 \leq k \leq d$ we have 
\begin{align}\label{hmt-fct-der} \frac{\partial \phi_n}{\partial x_k}(x) &= -\sqrt \frac{n_k+1}{2} \phi_{n+e_k}(x) + \sqrt\frac {n_k}{ 2} \phi_{n-e_k} (x)\\
\label{hmt-fct-mult} x_k \phi_n(x) &= \sqrt \frac{n_k+1}{2} \phi_{n+e_k}(x) + \sqrt\frac {n_k}{ 2} \phi_{n-e_k} (x). \end{align}
\end{lemm}

 \section{Proof of Theorem \ref{theo-intro-main}}\label{main-theo:sec}
 In this section we prove Theorem \ref{theo-intro-main}, showing that order $r$ isotropic pseudodifferential operators are the same as operators whose associated matrices are symbol matrices of order $r/2$.

\begin{proof}[Proof that $A \in \Psi^r_\text{iso}(\RR) \Rightarrow K^{(A)} \in SM^{r/2}(\NN_0)$]

We assume that $A \in \Psi^r_\text{iso}(\RR)$ is given. We will use $L^2$ boundedness of order 0 pseudodifferential operators to show that the matrix of $A$ is of order $r/2$ in $m$ and $n$. Then, an integration by parts argument will show that the matrix is rapidly decreasing off of the diagonal. Finally, we will show that applying the difference operator gives a matrix that is comparable to the matrix of an operator one order lower, which will complete the proof. 

It can be shown (as discussed in the introduction) that if $H = -\frac{\partial^2}{\partial x^2} + x^2$ is the harmonic oscillator then $(1+H)^{s/2} \in \Psi_\text{iso}^s(\RR)$ with principal symbol $(1+|x|^2+|\xi|^2)^{s/2}$. For a proof of this, see \cite{Shubin:PDO}, Theorems II.10.1 and II.11.2. We define 
\begin{equation} \label{eq:conj-H}A_t = (1+H)^{-rt/2} A (1+H)^{-r(1-t)/2}. \end{equation} 
Because $A_t \in \Psi_\text{iso}^0(\RR)$, it is bounded on $L^2(\RR)$. This implies that 
\begin{equation*}
|(K^{(A_t)})_{m,n}| = \left| \bangle{ A_t \phi_n, \phi_m}\right| \leq \Vert A_t\Vert_{\mathcal{L}(L^2(\RR))} < \infty.
\end{equation*} 

By Lemma \ref{mat-mult}, the matrix of a product of two operators is the product of the matrices of the operators. Thus, using the convention that repeated indices are summed, 
\begin{align*}(K^{(A_t)})_{m,n}&= \langle(1+H)^{-rt/2} \phi_k,\phi_m\rangle_{L^2(\TT^d))} (K^{(A)})_{k,j} \langle(1+H)^{-r(1-t)/2} \phi_j,\phi_n\rangle_{L^2(\TT^d))} \\ &= \delta_{k,m} (2+2n)^{-rt/2}(K^{(A)})_{k,j} \delta_{j,n} (2+2m)^{-r(1-t)/2}\\
&= (2+2n)^{-rt/2}(2+2m)^{-r(1-t)/2} (K^{(A)})_{m,n}\\
&= C (1+n)^{-rt/2}(1+m)^{-r(1-t)/2} (K^{(A)})_{m,n}.
\end{align*}
Combining this with the above bounds on $(K^{(A_t)})_{m,n}$ (because it is the matrix of a order 0 pseudodifferential operator), we thus have that 
\[|(K^{(A)})_{m,n} | \lesssim (1+n)^{rt/2}(1+m)^{r(1-t)/2}.\] 
If the order of $A$, $r$, is positive, taking $t=1/2$ gives
\begin{align*}|(K^{(A)})_{m,n}| &\lesssim (1+m)^{r/4}(1+n)^{r/4}\\
&= (1+m + n + mn)^{r/4}\\
&\leq ( 1 + 2m + 2n + 2mn)^{r/4}\\
&\leq ( 1 + 2m + 2n + 2mn+ m^2 + n^2)^{r/4}\\
& = (1+m+n)^{r/2}
\end{align*}
and if the order $r\leq0$, taking $t=0,1$ gives
\begin{align*}|(K^{(A)})_{m,n}| &\lesssim (1+\min(m,n))^{r/2}\\
& \leq (1+2^{-1}(m+n)^{r/2}\\
&\lesssim (1+m+n)^{r/2}.
\end{align*}

In order to show off diagonal decay of $K^{(A)}$, we first show how to simplify the integral expression of $(K^{(A)})_{m,n}$ by using the fact that Hermite functions are eigenfunctions of the Fourier transform:
\begin{align*}
(K^{(A)})_{m,n} & = \bangle {A \phi_n, \phi_m}\\
& = \int e^{i (x-y) \xi} a(x,\xi) \phi_n(y) \phi_m(x) dy\dbar\xi dx\\
& = \int e^{i x \xi} a(x,\xi) \mathcal{F}_{y\to \xi}[\phi_n](\xi)\phi_m(x) \dbar\xi dx\\
& = \int e^{i x \xi} a(x,\xi) [\sqrt{2\pi} (-i)^n \phi_n(\xi) \phi_m(x) \dbar\xi dx\\
& = \sqrt{2\pi} (-i)^n \int e^{i x\xi} a(x,\xi) \phi_n(\xi) \phi_m(x) \dbar\xi dx.
 \end{align*}
 Letting $H_x = x^2 - \frac{\partial^2}{\partial x^2}$ and $H_\xi = \xi^2 - \frac{\partial^2}{\partial\xi^2 }$, we thus have that 
\allowdisplaybreaks
\begin{align*}
 & 2 (m-n)  (K^{(A)})_{m,n} \\
 &= 2(m+1 - n-1)(K^{(A)})_{m,n}\\
& = \sqrt{2\pi}(-i)^n\int_{\RR^d\times\RR^d}e^{ix\cdot \xi} a(x,\xi) (H_x - H_\xi)
\left[\phi_n(\xi) \phi_m(x)\right] \ dx \dbar\xi\\
&= \sqrt{2\pi}(-i)^n\int_{\RR^d\times\RR^d}(H_x - H_\xi) \left[e^{ix\cdot \xi} a(x,\xi)\right]\phi_n(\xi) \phi_m(x) \ dx\dbar\xi.
\end{align*}
Now, we will show that the difference of the harmonic oscillators applied to $e^{ix\xi} a(x,\xi)$ results in $e^{ix\xi} \tilde a(x,\xi)$ where $\tilde a(x,\xi)$ is a new symbol, still of order $r$. 
\begin{align*}
(H_x - H_\xi) & \left[e^{ix\xi} a(x,\xi)\right]\\
& =  \left(x^2  - \xi^2 - \frac{\partial^2}{\partial^2 x} + \frac{\partial^2}{\partial^2\xi} \right) e^{ix\xi} a(x,\xi) \\
& =  \left(x^2  - \xi^2 +\xi^2 - x^2\right) e^{ix\xi} a(x,\xi)\\  & \qquad+ e^{ix\xi} \left[ - 2 i\xi  \frac{\partial}{\partial x} + 2 i x \frac{\partial}{\partial\xi} - \frac{\partial^2}{\partial^2 x} + \frac{\partial^2}{\partial\xi^2}  \right] a(x,\xi)\\
&=  e^{ix\xi} \left[ -2i\xi  \frac{\partial}{\partial x} +2 i x \frac{\partial}{\partial\xi} - \frac{\partial^2}{\partial^2 x} + \frac{\partial^2}{\partial^2\xi}  \right] a(x,\xi)\\
&= e^{ix\xi} \tilde a(x,\xi).
\end{align*}
Applying this repeatedly, we see that for any $N\geq 0$, multiplying $K_{m,n}^{(A)}$ by $(m-n)^N$ is the same as considering $(K^{(\tilde A)})_{m,n}$ for some $\tilde A \in \Psi_\text{iso}^r(\RR)$, and thus, repeating the boundedness argument given above for $\tilde A$, we thus have
 \begin{equation}\label{iso-no-diff-decay}|(K^{(A)})_{m,n}| \leq C_N (1+ |n-m|)^{-N}(1+m+n)^{r/2}.\end{equation}
 
It remains to be shown that the difference operator lowers the degree of the matrix. For the difference operator, notice that (letting $C : = \frac{1}{\sqrt{2}} (x+\partial_x)$ be the annihilation operator and $C^\dagger = \frac{1}{\sqrt{2}}(x- \partial_x)$ the creation operator)
\allowdisplaybreaks
\begin{align*}(\triangle K^{(A)})_{m,n} = & \langle A \phi_{n+1} ,\phi_{m+1} \rangle - \langle A \phi_n , \phi_m\rangle\\
= & \frac{1}{\sqrt{(m+1)(n+1)}}\langle A C^\dagger \phi_n , C^\dagger \phi_m\rangle - \langle A \phi_n ,\phi_m\rangle\\
=& \frac{1}{\sqrt{(m+1)(n+1)}}\langle C A C^\dagger \phi_n , \phi_m\rangle - \langle A \phi_n ,\phi_m\rangle\\
=& \frac{1}{\sqrt{(m+1)(n+1)}}\langle (A C + [C,A]) C^\dagger \phi_n , \phi_m\rangle - \langle A \phi_n ,\phi_m\rangle\\
=& \left(\frac{n + 1}{\sqrt{(m+1)(n+1)}} - 1\right) \langle A \phi_n ,\phi_m\rangle \\
&\quad+ \frac{1}{\sqrt{(m+1)(n+1)}}\langle ([C,A]) C^\dagger \phi_n , \phi_m\rangle\\
=& \left(\sqrt{\frac{n+1}{m+1}} - 1\right) \langle A \phi_n ,\phi_m\rangle \\
&\quad+ \frac{1}{\sqrt{(m+1)(n+1)}}\langle ([C,A]) C^\dagger \phi_n , \phi_m\rangle\\
=& \left(\frac{\sqrt{n+1}-\sqrt{m+1}}{\sqrt{m+1}}\right) \langle A \phi_n ,\phi_m\rangle\\&\quad + \frac{1}{\sqrt{(m+1)(n+1)}}\langle ([C,A]) C^\dagger \phi_n , \phi_m\rangle\\
=& \left(\frac{n-m}{\sqrt{m+1} (\sqrt{n+1} + \sqrt{m+1})}\right) \langle A \phi_n ,\phi_m\rangle \\&\quad+ \frac{1}{\sqrt{(m+1)(n+1)}}\langle ([C,A]) C^\dagger \phi_n , \phi_m\rangle.
 \end{align*}
 Because \[(\sqrt{m+1} + \sqrt{n+1})^2 = m+n+2 + 2\sqrt{(m+1)(n+1)} \geq (1+|m|+|n|)\] and \[ (m+1)(n+1) = (1 + m + n + mn) \gtrsim (1+m+n)^2,\] combining these inequalities with \eqref{iso-no-diff-decay} (the boundedness argument applies for $[C,A]C^\dagger$, which is a degree $r+1$ isotropic operator) we see that the difference operator will lower the degree in $m +n$ by $1$ and still preserve the off diagonal boundedness. Repeating this argument inductively for higher powers of the difference operator, we see that $K^{(A)}$ is a symbol matrix of order $r/2$. 

 \end{proof}
 
 \begin{proof}[Proof that $K^{(A)} \in SM^{r/2}(\NN_0)\Rightarrow A \in \Psi^r_\text{iso}(\RR)$]
 
We assume that $A: \cS(\RR) \to \cS'(\RR)$ is given, with $K^{(A)} \in SM^{r/2}(\NN_0)$. By multiplying by appropriate powers of the harmonic oscillator, we will show that we get a Hilbert-Schmidt operator. Using this, we will show that we have weighted $L^2$ bounds on the derivatives of the symbol of $A$. This will give weighted $L^\infty$ bounds on derivatives of the symbol, which will show that $A \in \Psi^{r+ \epsilon}_\text{iso}(\RR)$ for all $\epsilon>0$. Finally a sort of converse to $L^2$ boundedness of order 0 operators will show that we can take $\epsilon = 0$, as desired. 

Let $K(x,y)$ be the Schwartz kernel of $A$. We define a distribution
 \[a(x,\xi) = - \int_{\RR} e^{i(y-x)\cdot \xi} K(x,y)\ \dbar y.\] This definition gives that (where in the sequel we will follow the standard practice of writing an integral where we actually mean pairing with a distribution) 
  \[ K_{m,n}^{(A)} = \sqrt{2\pi}(-i)^n\int_{\RR\times\RR}e^{ix\cdot \xi} a(x,\xi) \phi_n(\xi) \phi_m(x)  \ dx\dbar\xi.\] 
  
  However, if we replace $a(x,\xi)$ with $\frac{\partial a}{\partial x}(x,\xi)$ we see from the following calculation that this results in a symbol matrix of order $(r-1)/2$ (we will use Lemma \ref{lem:hmt-fct-der-mult}, which gives a formula for derivatives of Hermite functions as well as the product of Hermite functions with a linear function):
  
  \allowdisplaybreaks
  \begin{align*}
&\sqrt{2\pi}  (-i)^n\int_{\RR\times\RR} e^{ix\cdot \xi} \frac{\partial a}{\partial x}(x,\xi) \phi_n(\xi) \phi_m(x) d\xi dx \\
&  = \sqrt{2\pi} (-i)^{n+1}\int_{\RR\times\RR} \left(\xi + i \frac{\partial}{\partial x} \right) \left [e^{ix\cdot \xi} a(x,\xi)\right]  \phi_n(\xi) \phi_m(x) d\xi dx\\
&  = \sqrt{2\pi} (-i)^{n+1}\int_{\RR\times\RR} e^{ix\cdot \xi} a(x,\xi) \left(\xi - i \frac{\partial}{\partial x} \right)  \phi_n(\xi) \phi_m(x) d\xi dx\\
&  = \sqrt{2\pi} (-i)^{n+1}\int_{\RR\times\RR} e^{ix\cdot \xi} a(x,\xi) \left( \sqrt \frac{n+1}{2} \phi_{n+1}(\xi) \phi_m(x) \right. \\
 & \qquad + \sqrt \frac n2 \phi_{n-1}(\xi) \phi_m(x) + i \sqrt \frac {m+1}{2} \phi_n(\xi) \phi_{m+1}(x)\\
 &\qquad  \left. - i \sqrt\frac m2 \phi_n(\xi) \phi_{m-1} (x) \right) d\xi dx\\
 &  = \sqrt\frac{n+1}{2} K_{m,n+1}^{(A)} - \sqrt\frac n2 K_{m,n-1}^{(A)} + \sqrt\frac{m+1}{2} K_{m+1,n}^{(A)} - \sqrt\frac m2 K_{m-1,n}^{(A)}\\
 &=  \sqrt\frac{n+1}{2} K_{m,n+1}^{(A)} - \sqrt\frac{n+1}{2} K_{m-1,n}^{(A)} +  \sqrt\frac{n+1}{2} K_{m-1,n}^{(A)} \\
 & \qquad - \sqrt\frac m2 K_{m-1,n}^{(A)} + \sqrt\frac n2 K_{m+1,n}^{(A)}  - \sqrt\frac n2 K_{m,n-1}^{(A)} \\
 & \qquad + \sqrt\frac{m+1}{2} K_{m+1,n}^{(A)} -\sqrt\frac n2 K_{m+1,n}^{(A)} \\
 & = \sqrt \frac{n+1}{2} \findif{} K_{m-1,n}^{(A)} + \left( \sqrt\frac{n+1}{2} - \sqrt\frac m2 \right) K_{m-1,n}^{(A)}\\
 & \qquad + \sqrt\frac n2 \findif{} K_{m,n-1}^{(A)} + \left( \sqrt\frac{m+1}{2} - \sqrt\frac n2 \right) K_{m+1,n}^{(A)}.
  \end{align*} 
It is clear that the order of symbol matrices is multiplicative and that the difference operator applied to a symbol matrix is a symbol matrix of one order lower, which shows that the terms with the difference operator are symbol matrices of order $(r-1)/2$.  Lemma \ref{sym-mat-mult} (which deals with multiplying a symbol matrix by a function which is not necessarily decaying off the diagonal, but does have appropriate diagonal decay) shows that the remaining terms are also symbol matrices of order $(r-1)/2$, so thhe sum of all four terms is also a symbol matrix of order $(r-1)/2$. A similar calculation holds for $\xi$ derivatives. 
  
  We let (for $\alpha,\beta\in \NN_0$) $A^{\alpha,\beta}$ be the operator with Schwartz kernel 
  \[\label{eq:box-long-comp} K^{\alpha,\beta} (x,y) = \int_{\RR} e^{i(x-y)\cdot \xi} \frac{\partial^{\alpha+\beta} a}{\partial x^\alpha \partial \xi^\beta}(x,\xi) \ \dbar\xi.\] 
  From our above results, we have $A^{\alpha,\beta} : \cS(\RR) \to \cS'(\RR)$ and its matrix is a symbol matrix of order $(r-\alpha-\beta)/2$. Now, for $s\in \RR$, we define $B^{\alpha,\beta}_s = (1+H)^{-s/2} A^{\alpha,\beta}$. By Lemma \ref{sym-mat-l2-memb} (which says that order $r<-1/2$ symbol matrices are $\ell^2$ summable, implying that the corresponding operator has an $L^2$ summable kernel) if $r-s-\alpha-\beta < -1$, then $B_s^{\alpha,\beta}$ is Hilbert-Schmidt because its kernel is square summable (see e.g. Theorem VI.23 in Reed-Simon \cite{ReedSimon:FA}). Thus, for the kernel of $B_{s}^{\alpha,\beta}$ we have $K_s^{\alpha,\beta} (x,y)\in L^2(\RR\times \RR)$. This implies that for $r(x,\xi)$ the symbol of $(1+H)^{s/2}$, 
  \begin{align*}A^{\alpha,\beta} f & = (1+H)^{s/2} \left(\int_{\RR} K_s^{\alpha,\beta}(y,z) f(z) \ dz\right)\\
  & = \int e^{i(x-y)\xi} r(x,\xi) K^{\alpha,\beta}_s(y,z) f(z)  \ dzdy\dbar\xi . \end{align*}
 Taking the Fourier transform \[ \frac{\partial^{\alpha+\beta} a}{\partial x^\alpha \partial \xi^\beta}(x,\xi)  \in (1+|x|^2 + |\xi|^2)^{-s/2} L^2(\RR\times \RR)\] for $r-s-\alpha-\beta<-1$, because the Fourier transform of $K^{\alpha,\beta}_s$ is in $L^2(\RR\times \RR)$, and the principal symbol of $(1+H)^{-s/2}$ is $(1+x^2  + \xi^2)^{-s/2}$ and the lower order terms will only improve the integrability. By Sobolev embedding, this implies that actually $a(x,\xi) \in C^\infty(\RR)$. 
 
We have weighted $L^2(\RR\times\RR)$ bounds on the derivatives of $a$, but to show that $a$ is a symbol, we would like weighted $L^\infty(\RR\times\RR)$ bounds on the derivatives. This is a standard argument, we include it in Appendix \ref{appB}. More precisely, Lemma \ref{L2-to-Linfty:theo} shows that for $r'$ with $r-r'-\alpha-\beta<0$
 \[ \frac{\partial^{\alpha+\beta} a}{\partial x^\alpha \partial \xi^\beta}(x,\xi)  \in (1+|x|^2 + |\xi|^2)^{r'/2} L^\infty(\RR\times \RR).\] Thus, we see that $A \in \Psi^{r'}_\text{iso} (\RR)$ for $r'>r$, and $B_s^{\alpha,\beta} \in \Psi^{r'-s - \alpha-\beta}_\text{iso}(\RR)$ for $r'>r$. Thus, to complete the proof, all that we must show is that $A \in \Psi_\text{iso}^r (\RR)$. Notice that this implies that for $ r -s = \alpha+\beta$, $B_{s}^{\alpha,\beta} \in \Psi_\text{iso}^{\epsilon_0} (\RR) $ for any $\epsilon_0>0$. However, since the symbol matrix of $B_{s}^{\alpha,\beta}$ is order $0$, by Lemma \ref{lem:ord0-bd} $B^{\alpha,\beta}_{s}$ is bounded on $L^{2}(\RR)$. Thus, we see that for $r-s =\alpha+\beta$, $B_s^{\alpha,\beta}$ is bounded on $L^2(\RR)$. Using the following lemma, we can thus conclude that the symbol of $B_s^{\alpha,\beta}$ is bounded, which then implies that $A \in \Psi_\text{iso}^r(\RR)$ because that will give the necessary bounds on each derivative. 
 \begin{lemm}
 For $Q \in \Psi_\text{iso}^{\epsilon_0} (\RR)$ for some $\epsilon_0<1/2$, if $Q$ is bounded on $L^2(\RR)$, then the symbol of $Q$, $\sigma_L(Q) = q(x,\xi)$ is bounded, i.e. \[ \sup_{(x,\xi) \in \RR^2} |q(x,\xi) | < \infty\]
 \end{lemm}

The proof of this is a calculation based on the idea of letting a pseudodifferential operator act on an exponent to gain control of the symbol. We give it in Appendix \ref{appC}.
 \end{proof}

This theorem does not hold as stated in $\RR^d$ for $d\geq2$. For example, consider $C^\dagger_1$, the creation operator in the first coordinate. Recall that \[ C^\dagger_1 \phi_n = \sqrt{n_1+1} \phi_{n+e_1} \] Thus \[ (K^{(C^\dagger_1)})_{m,n} = \langle C_1^\dagger \phi_n, \phi_m\rangle = \sqrt{n_1+1}\delta_{n+e_1,m} \] This has the appropriate decay; away from the diagonal it is zero, and for $n+e_1=m$ it is order $1/2$ in $|n|$, as in the $d=1$ case. However, taking the difference operator in the first coordinate gives \begin{align*}(\findif{e_1}K^{(C^\dagger_1)})_{m,n} &= (\sqrt{n_1+3} - \sqrt{n_1+1})\delta_{n+e_1,m}\\
& = \frac{2}{\sqrt{n_1+3} + \sqrt{n_1+1}}\delta_{n+e_1,m}\end{align*}
This is order $-1/2$ in $n_1$, but merely bounded for e.g. $n_2$, so in general the order in $|n|$ is only $0$. Furthermore, taking further difference operators in the first coordinate do not change the order in $|n|$ at all. 

It is possible to prove an analogous theorem for dimension $d\geq2$ by using arguments similar to the above proof. 
\begin{theo}\label{theo:iso-high-dim}
For $A$ a continuous linear operator $\cS(\RR^{d})\to \cS'(\RR^{d})$, $A \in \Psi^r_\text{iso}(\RR^d)$ if and only if $(K^{(A)})_{m,n}$ has the following property for all multi-indexes $\alpha,\beta \in \NN_0^d$
\begin{equation}\label{eq:iso-gen-dim} | (\square^{\alpha,\beta} K^{(A)})_{m,n}|  \leq (1+|m|+|n|)^{(r-|\alpha| - |\beta|)/2}  \end{equation} where we define $(\square_{x_k} K)_{nm} $ by 
\begin{equation*}  \sqrt\frac{n_k+1}{2} K_{n+e_k,m}  - \sqrt\frac{n_k}{2}K_{n-e_k,m}+ \sqrt \frac{m_k+1}{2} K_{n,m+e_k} -  \sqrt \frac{ m_k}{2} K_{n,m-e_k}\end{equation*}
 and $(\square_{\xi_k} K)_{nm} $ by 
 \begin{equation*}  \sqrt \frac{m_k+1}{2} K_{n,m+e_k} + \sqrt \frac{m_k}{2} K_{n,m-e_k} - \sqrt\frac{n_k+1}{2} K_{n+e_k,m} - \sqrt\frac{n_k}{2} K_{n-e_k,m} \end{equation*} 
and let $\square^{\alpha,\beta} = \left(\square_{x_1}\right)^{\alpha_1}  \cdots  \left(\square_{x_d}\right)^{\alpha_d}  \left(\square_{\xi_1}\right)^{\beta_1}  \cdots  \left(\square_{\xi_d}\right)^{\beta_d} $. 
\end{theo}

The $\square^{\alpha,\beta}$ operator is really the matrix analogue of applying $\partial_x^\alpha\partial_\xi^\beta$ to the symbol of $A$, as we illustrate in the case of $(\alpha,\beta) = (e_{k},0)$:
\allowdisplaybreaks\begin{align*}
& (\square_{x_k} K)_{nm} \\
& =  \sqrt\frac{n_k+1}{2} K_{n+e_k,m}  - \sqrt\frac{n_k}{2}K_{n-e_k,m}+ \sqrt \frac{m_k+1}{2} K_{n,m+e_k} -  \sqrt \frac{ m_k}{2} K_{n,m-e_k}\\
 &=(2\pi)^{d/2}(-i)^{n+1} \int_{\RR^d}\int_{\RR^d} e^{i x\cdot \xi} a(x,\xi) \left( \sqrt\frac{n_k+1}{2} \phi_{n+e_k} \phi_m 
\right. \\ &\ \ \ \  \left. 
 + \sqrt\frac{n_k}{2}\phi_{n-e_k} \phi_m + i \sqrt \frac{m_k+1}{2} \phi_n\phi_{m+e_k} - i\sqrt \frac{ m_k}{2} \phi_n \phi_{m-e_k}\right)d\xi dx\\
 &=(2\pi)^{d/2}(-i)^{n+1} \int_{\RR^d}\int_{\RR^d} e^{i x\cdot \xi} a(x,\xi)\left( \xi_k-  i \frac{\partial}{\partial x_k}\right) \phi_n(\xi) \phi_m(x)d\xi dx\\
  &=(2\pi)^{d/2}(-i)^{n+1} \int_{\RR^d}\int_{\RR^d} \left( \xi_k + i \frac{\partial}{\partial x_k}\right) e^{i x\cdot \xi} a(x,\xi) \phi_n(\xi) \phi_m(x)d\xi dx\\
 &= (2\pi)^{d/2}(-i)^{n+2} \int_{\RR^d}\int_{\RR^d}  e^{i x\cdot \xi}\frac{\partial a}{\partial x_k}(x,\xi) \phi_n(\xi) \phi_m(x)d\xi dx
\end{align*}

The proof of Theorem \ref{theo:iso-high-dim} follows the one dimensional case almost identically, except that the above calculation replaces the difference operator calculations. For example, to show that for $A \in \Psi^{r}_{\text{iso}}(\RR^{d})$, equation \eqref{eq:iso-gen-dim} is satisfied, the argument for $(\alpha,\beta) = 0$ follows in exactly the same way from $L^{2}$ boundedness of zeroth order isotropic pseudodifferential operators. Then, applying the box operator gives a matrix which is the matrix of an isotropic pseudodifferential operator with the appropriate derivatives on the symbol, so again we have the desired bounds. Conversely, it is possible to show (as we discuss below) that matrices obeying \eqref{eq:iso-gen-dim} with $r < -d/2$ are square summable, and thus correspond to Hilbert-Schmidt operators. Thus, we can establish weighted bounds on the distribution that we would like to show is the symbol (exactly as in the one dimensional case). Finally, derivatives of the ``symbol'' correspond to applying $\square^{\alpha,\beta}$ to the matrix, and we thus have weighted bounds on the derivatives of the symbol, allowing us to conclude as in one dimension.

The conditions of \eqref{eq:iso-gen-dim} are somewhat unsatisfying, as they are considerably more complex than the simple difference operator results in the one dimensional case. For example, it can be shown that infinite matrices obeying \eqref{eq:iso-gen-dim} are rapidly decreasing off of the diagonal, which is not obvious from the condition. To see this, consider the following identity
\[(m_k - n_k)K_{nm} = \sqrt \frac {m_k}{2} (\square_{x_k} + \square_{\xi_k}) K_{n,m-e_k} - \sqrt\frac{n_k}{2} (\square_{x_k} - \square_{\xi_k}) K_{n-e_k,m}. \] 
It is not hard to see that the right hand side still obeys \eqref{eq:iso-gen-dim} with the same $r$, which shows that $K_{nm}$ is rapidly decaying off of the diagonal (in particular, this shows that condition \eqref{eq:iso-gen-dim} with $r<-d/2$ implies that the matrix is square summable, as promised above). 

It is likely that the issues relating to the difference operator in higher dimensions is a consequence of the high multiplicity of the harmonic oscillator eigenspaces. It seems possible that a better result could be obtained through a ``rearranging'' of the eigenspaces. That is, the difference operator compares the operator's action on $\phi_n$ and $\phi_{n+e_k}$, elements of the $|n|+d$ and the $|n|+d+1$ eigenspaces, but there is no reason that these are the proper elements from these two eigenspaces to compare. However, we have been unable to find a satisfactory manner in which to compare the two eigenspaces in order to obtain a simpler condition on the matrix.

 \section{An Isotropic Beals Theorem} 
 We begin by recalling the classical theorem of Beals characterizing (standard) pseudodifferential operators on $\RR^{d}$. For an operator $A: \cS(\RR^{d}) \to \cS'(\RR^{d})$, we define $L_{k} (A) := [x_{k},A]$ (where $[B,C] := BC-CB$ is the commutator) as well as $R_{k}(A) := [ \partial_{k}, A]$. Theorem 1.4 in \cite{Beals:CommCharPSIDO} gives
   
 \begin{theo}[Beals commutator characterization] A continuous linear map \[A: \cS(\RR^{d}) \to \cS'(\RR^{n})\] is a pseudodifferential operator of order $r$, i.e.\ $A\in \Psi^{r}(\RR^{d})$ if and only if for $\alpha,\beta \in (\NN_{0})^{d}$ we have that 
 \[ L^{(\alpha)} ( R^{(\beta)} (A) ) \in \cL(H^{r+|\beta| - |\alpha|}(\RR^{d} ) \to L^{2}(\RR^{d})).\] Here, we have written $\cL(H^{r+|\beta| - |\alpha|}(\RR^{d}) \to L^{2}(\RR^{d}))$ for bounded linear maps from the $r+|\beta|-|\alpha|$ Sobolev space to $L^{2}$ and have written $L^{(\alpha)}$ for $L$ composed with itself $\alpha$ times (similarly for $R^{(\beta)}$).
 
 \end{theo}
 
 By using Theorem \ref{theo-intro-main}, we can give a version of commutator characterization in the setting of isotropic pseudodifferential operators. First, we recall the definition of isotropic Sobolev spaces
 \begin{align*}
 H^{s}_{\text{iso}}(\RR)  & = \{ f \in \cS'(\RR) : (1+H)^{s/2} f \in L^{2}(\RR) \} \\
 & = \left \{ f \in \cS'(\RR) : \sum_{k=0}^{\infty} (1+k)^{s} \bangle{f,\phi_{k}}_{L^{2}(\RR)}^{2} < \infty \right\}
 \end{align*}
 
 We define the operator $Z:\cS(\RR)\to \cS(\RR)$ by $\phi_{k}\mapsto \phi_{k-1}$ (extending linearly). It is clear that $Z^{\dagger}(\phi_{k}) = \phi_{k+1}$. Furthermore, notice that the matrix of $Z$ is 
 \[ K^{(Z)} = \left(\ \begin{matrix}
  0 \\
 1 & 0 \\
  & 1 & 0 \\
   & & 1& 0\\
   &&& \ddots
 \end{matrix}\ \right)
 \]
 and thus, by Theorem \ref{theo-intro-main}, $Z \in \Psi^{0}_{\text{iso}}(\RR)$ (but it is certainly not a differential operator). It is clear that $ZZ^{\dagger} = Id$. Also, recall the harmonic oscillator $H = -\partial_{x}^{2} + x^{2}$ has matrix 
  \[ K^{(H)} = \left(\ \begin{matrix}
  1 \\
  & 3 \\
  &  & 5 \\
   & & & 7\\
   &&& & \ddots
 \end{matrix}\ \right)
 \]
Given these two operators, we define $\tilde Z(A) := [A,Z]$ and $\tilde H(A) : =[A,H]$. Now, we recall our version of the isotropic Beals theorem and give the proof:
\newtheorem*{thm:associativity}{Theorem \ref{theo:iso-beals}}
\begin{thm:associativity}
An operator $A :\cS(\RR) \to \cS'(\RR)$ with $A \in \Psi^{r}_{\text{iso}}(\RR)$ if and only if for all $\alpha,\beta \in \NN_{0}$ and $s\in \RR$
\begin{equation}\label{eq:iso-comm-cond} \tilde H^{(\alpha)}(\tilde Z^{(\beta)}(A)) \in \cL(H^{r+s-2\beta}_{\text{iso}}(\RR) \to H^{s}_{\text{iso}}(\RR)). \end{equation}
\end{thm:associativity}
\begin{proof}
It is enough to show that \eqref{eq:iso-comm-cond} for all $\alpha,\beta\in\NN_{0}$ implies that $A \in \Psi^{r}_\text{iso}(\RR)$. We will show that $K^{(A)}$ is a symbol matrix of the appropriate order, $K^{(A)} \in SM^{r/2}(\NN_{0})$, and then use Theorem \ref{theo-intro-main} to conclude the desired result. 

First, notice that by repeating the argument following equation \eqref{eq:conj-H} (conjugating by the appropriate powers of the harmonic oscillator) we have that if an operator $B :\cS(\RR)\to\cS'(\RR)$ with $B  \in \cL(H^{r+s}_{\text{iso}}(\RR) \to H_\text{iso}^{s}(\RR))$ for all $s\in \RR$ and some $r \in \RR$, then 
\[ |(K^{(B)})_{m,n} | \lesssim (1+m+n)^{r/2} .\]
In particular, taking $B=A$ in the above, shows that $K^{(A)}$ has the desired overall decay of order $r/2$. It thus remains to show that the finite difference operator reduces this order of decay and off diagonal decay, following assumption \eqref{eq:iso-comm-cond}. For the finite difference operator, notice that 
\begin{align*}
(\findif{} K^{(A)})_{m,n} &  =\bangle{A\phi_{n+1},\phi_{m+1}}_{L^{2}(\RR)}  - \bangle{A\phi_{n},\phi_{m}}_{L^{2}(\RR)} \\
& = \bangle{AZ^{\dagger}\phi_{n},Z^{\dagger}\phi_{m}}_{L^{2}(\RR)}  - \bangle{A\phi_{n},\phi_{m}}_{L^{2}(\RR)}\\
& = \bangle{ZAZ^{\dagger}\phi_{n},\phi_{m}}_{L^{2}(\RR)}  - \bangle{A\phi_{n},\phi_{m}}_{L^{2}(\RR)}\\
& = \bangle{(AZ - [A,Z])Z^{\dagger}\phi_{n},\phi_{m}}_{L^{2}(\RR)}  - \bangle{A\phi_{n},\phi_{m}}_{L^{2}(\RR)}\\
& = \bangle{(A - [A,Z]Z^{\dagger})\phi_{n},\phi_{m}}_{L^{2}(\RR)}  - \bangle{A\phi_{n},\phi_{m}}_{L^{2}(\RR)}\\
& = -\bangle{ [A,Z]Z^{\dagger}\phi_{n},\phi_{m}}_{L^{2}(\RR)}.
\end{align*}
Because $K^{(Z^{\dagger})} \in SM^{0}(\NN_{0})$ is invertible, modulo a finite dimensional error (because $ZZ^{\dagger} = Id$ but $Z^{\dagger}Z (\phi_{k}) =\phi_{k}$ for $k>0$ and $0$ for $k=0$), \eqref{eq:iso-comm-cond} and the argument above with $B = \tilde Z(A)$ (along with the multiplicative property of kernel matrices, Lemma \ref{mat-mult}) implies that 
\[|(\findif{} K^{(A)})_{m,n} | \lesssim (1+m+n)^{r/2-1}. \] Repeating the above shows that 
\[|(\findif{\beta} K^{(A)})_{m,n} | \lesssim (1+m+n)^{r/2-\beta}. \] 

For the off diagonal decay, notice that for any $B  \in \cL(H^{s}_{\text{iso}}(\RR) \to L^{2}(\RR))$
\begin{align*}
2(m-n) (K^{(B)})_{m,n} & = (1+2m-1-2n)(K^{(B)})_{m,n} \\
& = \bangle{ B \phi_{n},(1+2m)\phi_{m}}_{L^{2}(\RR)}- \bangle{ B(1+2n) \phi_{n},\phi_{m}}_{L^{2}(\RR)}\\
& = \bangle{ B \phi_{n},H\phi_{m}}_{L^{2}(\RR)}- \bangle{ BH \phi_{n},\phi_{m}}_{L^{2}(\RR)}\\
& = \bangle{ HB \phi_{n},\phi_{m}}_{L^{2}(\RR)}- \bangle{ BH \phi_{n},\phi_{m}}_{L^{2}(\RR)}\\
& = - \bangle{ [B,H] \phi_{n},\phi_{m}}_{L^{2}(\RR)}.
\end{align*}
Thus, repeating this inductively, and combining with the above argument, we see that for all $N \geq 0$
\[ (m-n)^{N} |(\findif{\beta} K^{(A)})_{m,n} | \lesssim (1+m+n)^{r/2-\beta}\] 
which allows us to conclude $K^{(A)} \in SM^{r/2}(\NN_{0})$, as desired. 
\end{proof}
  
 \appendix 
 \section{Symbol Matrix Properties}\label{appA}
 In this appendix we compile various properties about matrices of operators and symbol matrices that were used in the preceding sections. First, we show that associating a map with its matrix is multiplicative in the same sense as finite dimensional linear operators. 
 \begin{lemm}\label{mat-mult}
The map from operators $A:\cS(\RR) \to \cS(\RR)$ to matrices $K^{(A)}: \NN_0\ \times \NN_0 \to\RR$ is a homomorphism where we define the product of two matrices to be the infinite dimensional analogue of matrix multiplication. Namely, for continuous linear $A,B: \cS(\RR)\to \cS(\RR)$ \[(K^{(AB)})_{m,n} = \sum_{k\in\NN_0} (K^{(A)})_{m,k} (K^{(B)})_{k,n}.\]
\end{lemm}
\begin{proof} For any Schwartz function $f \in \cS(\RR)$, we have that the following sum converges in $\cS(\RR)$: \[\sum_{0\leq k \leq N} \phi_k \langle f, \phi_k \rangle_{L^2(\RR)} \to f\] as $N\to\infty$. Now, using this fact, along with the fact that $A,B$ are continuous as operators on $\cS(\RR)$, the following calculation is justified:
\begin{align*}(K^{(AB)})_{m,n} & = \bangle{ AB\phi_n ,\phi_m}_{L^2(\RR)} \\
& = \bangle{ A\left(\sum_{k\in \NN_0} (K^{(B)})_{k,n}\phi_k \right) ,\phi_m}_{L^2(\RR)} \\
& = \bangle{ \sum_{k\in \NN_0} (K^{(B)})_{k,n} A(\phi_k)  ,\phi_m}_{L^2(\RR)} \\
& =  \sum_{k\in \NN_0} \bangle{ (K^{(B)})_{k,n} A(\phi_k)  ,\phi_m}_{L^2(\RR)}\\
& =  \sum_{k\in \NN_0} \bangle{ A\phi_k,\phi_m}_{L^2(\RR)} (K^{(B)})_{k,n} \\
& =  \sum_{k\in \NN_0} (K^{(A)})_{m,k}(K^{(B)})_{k,n}. 
 \end{align*}
 \end{proof}

Now we show that symbol matrices are closed under multiplying by a function which is polynomially bounded with the difference operator lowering the decay by one order. 
\begin{lemm}\label{sym-mat-mult}
If $f(m,n)$ is a symbol matrix of order $r$ and $g(m,n)$ is any function $g:\NN_0\times\NN_0 \to \RR$ satisfying  \[|\left(\findif{\gamma} g\right)(m,n)| \leq C_{\gamma} (1+|m|+|n|)^{s-|\gamma|}\] for $\gamma \in \NN_{0}$ and $s \in \RR$, then $f(m,n) g(m,n)$ is a symbol matrix of order $r+s$. 
\end{lemm} 
\begin{proof}
It is easy to see that from the conditions on $f(m,n)$ and $g(m,n)$ we have \[|[fg](m,n)|\leq C'_{N,0} (1+|n-m|)^{-N}(1+|n|+|m|)^{r+s},\] so the only thing that is left to check is that applying the difference operator gives improved decay along the diagonal, which follows from
\[\findif{} \left[ fg\right](m,n)= \findif{} \left[f\right](m,n) \ g(m+1,n+1) + f(m,n) \findif{} \left[g\right](m,n).\]
Repeating the above calculation finishes the proof. 
\end{proof}

Because of the rapid decay of a symbol matrix off of the diagonal, we expect many of the properties of a diagonal matrix to hold for a symbol matrix. One such property is $\ell^2(\NN_0\times \NN_0)$ membership for symbol matrices:

\begin{lemm}\label{sym-mat-l2-memb}
If $r < - 1/2$, and $K \in SM^r(\NN_0)$ is any symbol matrix of order $r$  then $\{(K)_{m,n}\} \in \ell^2(\NN_0\times\NN_0)$. 
\end{lemm}
\begin{proof}
For a fixed $0<s<1$ and $N \in \NN_0$, we bound the sum as follows \allowdisplaybreaks
\begin{align*}
\sum_{m,n\in \NN_0} |(K)_{m,n}|^2=& \sum_{n \in \NN_0, k \geq -n} |(K)_{n+k, n}|^2\\
=& \sum_{|k| \leq |n|^s, k\geq -n}  |(K)_{n+k,n}|^2+  \sum_{|k| > |n|^s, k \geq -n} |(K)_{n+k,n}|^2\\
\lesssim &  \sum_{|k| \leq |n|^s, k \geq -n}  (1+n + n+k)^{2r}\\
&\qquad +  \sum_{|k| > |n|^s} (1+|k|)^{-2N}(1+n + n+k)^{2r}\\
\lesssim&\sum_{|k| \leq |n|^s, k\geq -n}  (1+n)^{2r} +  \sum_{|k| > |n|^s, k \geq -n} (1+|k|)^{-2N}(1+|k|)^{2|r|/s}\\
\lesssim&\sum_{n \in \NN_0} \#\{k: |k| \leq |n|^s\}  (1+n)^{2r}\\
& \qquad+  \sum_{n \in \NN_0} \int_{|x| > |n|^s} (1+|x|)^{-2N+ 2r/s}\ dx\\
\lesssim&   \sum_{n \in \NN_{0}} |n|^{s} (1+n)^{2r} \\
&\qquad+  \sum_{n \in \NN_{0}} \int_{ |n|^s}^\infty  (1+t)^{-2N+2r/s}\ dt\\
\lesssim  & \sum_{n \in \NN_{0}} |n|^{s}  (1+n)^{2r} \\
&\qquad+ \sum_{n \in \NN_{0}} (1+|n|^s)^{1-2N + 2r/s}.
\end{align*} 
For $s + 2r < -1$ and $s(1-2N) + 2r<-1$ the final terms will be finite. Because $r<-1/2$, we can choose $s>0$ such that $2r < -1-s$. For this $s$, we can then choose an $N$ large enough such that $s(1-2N) < -1 -2 r$. This completes the proof. 
\end{proof}

Given an arbitrary symbol matrix, it defines an operator on Schwartz functions, as is shown in the next lemma. 

\begin{lemm} \label{sym-mat-cont}
If $f(m,n)$ is a symbol matrix of order $r$, then it defines a continuous operator $\cS(\RR)\to\cS(\RR)$. 
\end{lemm} 
\begin{proof}
It is clear that our desired result is equivalent to proving that the operator
 \[ F: \{x_k\}_{k\in \NN_{0}} \mapsto \sum_{k \in \NN_0} f(n,k) x_k\] is a continuous linear operator on the space of rapidly decreasing sequences \[s(\NN_0) = \{ \{x_k\}_{k\in\NN_0} : |(1+|k|)^N x_k| \leq C_N \ \forall N\in\NN_0\}\] where the topology on $s(\NN_0)$ is given by the seminorms \[\Vert x_k \Vert_N = \sup_k |(1+k)^Nx_k|.\] 

Because a symbol matrix of negative order is also a symbol matrix of zero order, we may assume that $r\geq 0$. 
Notice that multiplication by $(1+|k|)^s$ is a continuous linear bijection $s(\NN_0) \to s(\NN_0)$. Denote this operator by $L_s$. Choosing any $N$, notice that 
\begin{align*}
& \Vert L_s\circ  F\circ L_l (\{x_k\})\Vert_N\\
 & = \left\Vert \sum_{k\in\NN_0} (1+|n|)^{s} f(n,k) (1+|k|)^{l}x_k\right\Vert_N \\
& = \sup_n (1+|n|)^{N+s} \left\vert\sum_{k\in\NN_0} f(n,k) (1+|k|)^l x_k \right\vert \\
& \leq C \sup_n (1+n)^{N+s} \sum_{k\in\NN_0} (1+|n|+|k|)^r (1+|k|)^l |x_k | \\
& \leq C\Vert x_k\Vert_{0} \sup_n (1+|n|)^{N+s} \sum_{k\in\NN_0} (1+|n|+|k|)^r (1+|k|)^{l}\\
& \leq C\Vert x_k\Vert_{0} \sup_n (1+|n|)^{N+s} \sum_{k\in\NN_0}(1+|n|)^r(1+|k|)^r (1+|k|)^{l}\\
& = C\Vert x_k\Vert_{0} \sup_n (1+|n|)^{N+s + r} \sum_{k\in\NN_0} (1+|k|)^{l+r}.
\end{align*} 
In particular, taking $l < -1-r$, and $s< -N -r$, we have that \[\Vert L_s \circ F \circ L_l (\{x_k\}) \Vert_N\leq C_{l,s} \Vert x_k \Vert_0.\]
Thus, for these $s,l$ we have that \begin{align*} \Vert F(\{x_k\}) \Vert_N & = \Vert (L_{-s} ( L_s \circ F \circ L_l) )( L_{-l} (\{x_k\})\Vert_N\\ 
& \leq \Vert L_{-s} \Vert_{\mathscr{L}(s(\NN_0)} \Vert L_s \circ F \circ L_s (L_{-l}(\{x_k\}) \Vert_N\\
&\leq C_{l,s} \Vert L_{-s} \Vert_{\mathscr{L}(s(\NN_0))} \Vert L_{-l} \Vert_{\mathscr{L}(s(\NN_0))} \Vert x_k\Vert_0.
\end{align*}
\end{proof}

Furthermore, an order 0 symbol matrix considered as an operator on functions, as in the previous lemma, extends to a bounded operator on $L^2(\RR)$
\begin{lemm}\label{lem:ord0-bd}
Considering an order 0 symbol matrix $f(m,n)$ as an operator \[A: \cS(\RR) \to \cS(\RR)\] (as in the previous lemma), $A$ extends to a bounded operator \[A:L^2(\RR)\to L^2(\RR).\]
\end{lemm}
\begin{proof}
This follows readily from Schur's test (for example, see Theorem 5.2 in \cite{Halmos:IntOp}) but for completeness we will give the proof. Our proof is essentially the same as the proof in \cite{Halmos:IntOp}, except in far less generality.

For $g \in L^2(\RR)$, we have that $\hat g(k) := \langle g, \phi_k\rangle_{L^2(\RR)}$ is in $\ell^2(\NN_0)$. Thus 
\begin{align*}
\sum_{n \in \NN_0} &  \left( \sum_{k\in\NN_0} f(n,k) \hat g(k) \right)^2\\
 & \leq \sum_{n \in \NN_0} \left( \sum_{k\in\NN_0} |f(n,k)|| \hat g(k) |\right)^2\\
&  = \sum_{n \in \NN_0} \left( \sum_{k\in\NN_0} \sqrt{|f(n,k)|}\sqrt{|f(n,k)|}| \hat g(k) |\right)^2\\
&  \leq  \sum_{n \in \NN_0} \left( \sum_{k\in\NN_0}{|f(n,k)|} \right) \left(\sum_{k\in \ZZ^d}{|f(n,k)|}| \hat g(k) |^2\right)\\
& \leq C_N^2  \sum_{n \in \NN_0} \left( \sum_{k\in\NN_0} (1+|n-k|)^{-N} \right) \left(\sum_{k\in \ZZ^d} (1+|n-k|)^{-N}| \hat g(k) |^2\right)\\
& \leq C_N^2  \left( \sum_{k\in\NN_0} (1+|k|)^{-N} \right) \left( \sum_{n,k \in \NN_0}(1+|n-k|)^{-N}| \hat g(k) |^2 \right)\\
& = C_N^2 \left( \sum_{k\in\NN_0}(1+|k|)^{-1} \right)^2 \left( \sum_{k \in \NN_0} |\hat g(k)|^2 \right).
\end{align*} 
\end{proof}
 
 \section{Weighted $L^2$ bounds into weighted $L^\infty$ bounds}\label{appB}
 
 Now, we will prove the following lemma. It is stated in \cite{Melrose:MAnotes}, but given as an exercise. We include the proof for completeness. 
\begin{lemm} \label{L2-to-Linfty:theo} 
For $f\in C^\infty (\RR^d)$, if \[\frac{\partial^\alpha f}{\partial x^\alpha}(x) \in (1+|x|^2)^{(s-|\alpha|)/2} L^2(\RR^d)\] for all $\alpha \in \ZZ^d$, then for $s'>s-d/2$ and any $\alpha \in \ZZ^d$ \[\frac{\partial^\alpha f}{\partial x^\alpha}(x) \in (1+|x|^2)^{(s'-|\alpha|)/2} L^\infty(\RR^d).\]
\end{lemm}
\begin{proof}
We may assume that $|x|>1$ on $\supp f$, because only the asymptotic behavior as $|x| \to \infty$ matters for either property, as $f$ is smooth. We introduce polar coordinates on $\{|x|>1\} \subset \RR^d$ \[ x = t \omega \text{ for } \omega \in S^{d-1} \text{ and } t>1 .\] In polar coordinates, our assumption on $f$ can be thus written as  \[\frac{\partial^\alpha f}{\partial x^\alpha} f(x) \in t^{s-|\alpha|} L^2(\RR^+\times S^{d-1}, t^{d-1} dt d\omega ).\] 
It is clear that $x_j \in t C^\infty(S^{d-1})$, so we see that \[\frac{\partial f}{\partial \omega_k} = \sum_{j=1}^d\frac{\partial x_j}{\partial \omega_k} \frac{\partial f}{\partial x_j} \in t^s L^2(\RR^+\times S^{d-1}, t^{d-1} dt d\omega ),\] because differentiating in $x_j$ lowers the power of $t$ by one, but a power of $t$ is also regained from the $\frac{\partial x_j}{\partial \omega_k}$, and multiplying by a function in $C^\infty(S^{d-1})$ cannot hurt the $L^2(S^{d-1})$ bounds. Similarly
\[\frac{\partial f}{\partial t} = \sum_{j=1}^d \frac{\partial x_j}{\partial t} \frac{ \partial f}{\partial x_j}  \in t^{s-1} L^2(\RR^+\times S^{d-1} , t^{d-1} dtd\omega). \]
Thus, we have that for any differential operator on $S^{d-1}$, $P$, and any $l \in \NN_0$ the assumption on $f$ gives \[ \frac{\partial^l}{\partial t^l} (P f(t,\omega))  \in t^{s-l} L^2(\RR^+\times S^{d-1}, t^{d-1} dt d\omega ).\] We can rewrite this as being integrable on $\RR^+$ with values in $L^2(S^{d-1})$\[ \frac{\partial^l}{\partial t^l} (P f(t,\omega))  \in t^{s-l - (d-1)/2} L^2(\RR^+; L^2 (S^{d-1})).\] This does not depend on the exact form of $P$, or even its order, so we can take $P$ to be very large order and elliptic, and thus for any $k$, we can consider original $f$ as a $L^2$ function on $\RR^+$ taking values in $H^k(S^{d-1})$ \[ \frac{\partial^l}{\partial t^l} ( f(t,\omega))  \in t^{s-l -(d-1)/2} L^2(\RR^+; H^k(S^{d-1})).\]Thus, for any differential operator on $S^{d-1}$, $Q$ of order $q$ we have \[ \frac{\partial^l}{\partial t^l} (Q f(t,\omega))  \in t^{s-l - (d-1)/2} L^2(\RR^+; H^{k-q}(S^{d-1})).\] For large enough $k$, by Sobolev embedding $H^{k-q} (S^{d-1}) \subset L^\infty(S^{d-1})$, so we have that 
 \[ \sup_{w \in S^{d-1}} \left\vert \frac{\partial^l}{\partial t^l} (Q f(t,\omega))  \right\vert \in t^{s-l - (d-1)/2} L^2(\RR^+).\] 
 Writing  \[ \sup_{w \in S^{d-1}} \left\vert \frac{\partial}{\partial t} \left(t^p \frac{\partial^{l}}{\partial t^{l}} (Q f(t,\omega)) \right) \right\vert = t^{s-l  + p-1 - (d-1)/2} g(t),\] for $g(t) \in L^2(\RR^+)$. Because $L^\infty$ membership for smooth functions does not depend on its values on any compact set, $g$ is zero for $t<2$.  Furthermore, \[ t^p \frac{\partial^l}{\partial t^l} (Q f(t,\omega)) = \int_{0}^t  \frac{\partial}{\partial t' } \left( t'^p\frac{\partial^{l+1}}{\partial {t'}^{l+1}} (Q f(t',\omega)) \right)dt'\] so 
 \begin{align*}& \sup_{t \in \RR^+, w\in S^{d-1}}\left\vert t^p\frac{\partial^l}{\partial t^l} (Q f(t,\omega)) \right\vert \\
 &\leq  \int_{0}^\infty \left\vert \frac{\partial}{\partial t'} \left( {t'}^{p} \frac{\partial^{l}}{\partial {t'}^{l}} (Q f(t',\omega))\right)\right\vert dt'\\
  &=  \int_{1}^\infty {t'}^{s-l-(d-1)/2+p-1} \left\vert g(t')  \right\vert dt'\\
   &\leq\left(  \int_{1}^\infty {t'}^{2s-2l-d-1 + 2p} \ dt'\right)^{1/2} \left( \int_1^\infty \vert g(t') \vert\  dt' \right)^{1/2}.
 \end{align*} This is finite for $2s-2l -d -1 + 2p < -1$, or in other words $s - l - d/2 + p < 0$. Reverting to regular coordinates, this proves the lemma. 
\end{proof}
 
 \section{Boundedness on $L^2$ implies a bounded symbol}\label{appC} 

 In this appendix we prove that given a isotropic pseudodifferential operator of order close enough to zero which is bounded on $L^2$ necessarily has a bounded symbol. The proof basically involves the idea of a pseudodifferential operator acting on an exponent, but we could not find this particular setting in the literature.
  \begin{lemm}
 For $Q \in \Psi_\text{iso}^{\epsilon_0} (\RR)$ for some $\epsilon_0<1/2$, if $Q$ is bounded on $L^2(\RR)$, then the symbol of $Q$, $\sigma_L(Q) = q(x,\xi)$ is bounded, i.e. \[ \sup_{(x,\xi) \in \RR^2} |q(x,\xi) | < \infty.\]
 \end{lemm}

\begin{proof}
Because $Q$ is bounded on $L^2(\RR)$ there is some $C>0$ such that \[ \Vert Q f \Vert_{L^2(\RR)} \leq \sqrt{C} \Vert f\Vert_{L^2(\RR)}.\] Choose $\epsilon>0$, $z_0\in \RR$ and $\varphi(x) \in C_0^\infty (\RR)$ a bump function with $\varphi(x)\in[0,1]$ and $\varphi \equiv 1$ for $|x|<\epsilon/2$ and $\varphi\equiv 0$ for $|x| > \epsilon$. Let $\varphi_{z_0} (x) = \varphi(x-z_0)$. For $v \in [-1,1]$ we thus have that \[ \Vert Q(e^{iy^2 v} \varphi_{z_0}(y)) \Vert_{L^2(\RR)}^2 \leq C \Vert \varphi_{z_0} \Vert_{L^2(\RR)}^2=C \Vert \varphi \Vert_{L^2(\RR)}^2.\] Thus, we have that \begin{equation}\label{eq:def-I} I_v(z_0) := \bangle{e^{-ix^2 v} Q^*Q(e^{iy^2v}\varphi(y)),\varphi}_{L^2(\RR)} \leq C \Vert \varphi \Vert_{L^2(\RR)}^2. \end{equation} Letting $p(x,\xi) = \sigma_L(Q^*Q)$ 
\begin{align*}
I_v(z_0) & = \int e^{i(x-y)\xi} p(x,\xi) e^{i(y^2-x^2)v} \varphi_{z_0}(y) \varphi_{z_0}(x)  \ dy\dbar\xi dx\\
&=\int e^{i(x-y)\xi} p(x+z_0,\xi) e^{i(y^2-x^2)v}e^{i(y-x)2vz_0}\varphi(y) \varphi(x)  \ dy\dbar\xi dx\\
& = \int e^{i(x-y)\xi} p(x+z_0,\xi + 2vz_0) e^{i(y^2-x^2)v} \varphi(y)\varphi(x)  \ dy\dbar\xi dx\\
& = \int e^{i(x-y)\eta} p(x+z_0,2vx+ 2vz_0+\eta) e^{i(y-x)^2v} \varphi(y)\varphi(x)  \ dy\dbar\eta dx.
\end{align*}
In the last line, we made the substitution $\xi = 2vx + \eta$. Now, expanding around $\eta = 0$ \begin{align*} p(x+z_0 &,2vx+2vz_0 + \eta)\\
& = p(x+z_0, 2vx + 2vz_0) + \int_0^1 p_2(x+z_0,2vx+2vz_0 + t\eta) \eta \ dt.\end{align*}
Combining this with the above, the first term simplifies via Fourier inversion giving\[ \int p(x+ z_0, 2vx + 2vz_0) \varphi(x)^2  \ dx.\]
Now, we will show that as $|z_0| \to \infty$ the second term goes to zero. Let $\chi \in C_0^\infty$ be a bump function $\chi(\eta) \in [0,1]$, $\chi(\eta) \equiv 1$ for $|\eta| < 1$ and $\chi(\eta) \equiv 0$ for $|\eta|>2$. Using this we can estimate the second term where $\eta$ is small \begin{align*}
 & \Big| \int \int_{t=0}^{t=1} e^{i(x-y)\eta} p_2(x+z_0,2vx+ 2vz_0+t\eta)\\
 & \qquad \qquad\qquad \qquad\qquad \qquad \eta \chi(\eta)e^{i(y-x)^2v} \varphi(y)\varphi(x) \ dt dy\dbar\eta dx\Big|\\
 & \leq \int \int_0^1 |p_2(x+z_0,2vx+2vz_0+t\eta)| |\eta| |\chi(\eta)| |\varphi(y)||\varphi(x)| \ dtdy\dbar\eta dx\\
  & \leq  C \int \int_0^1 (1+|x+z_0| + |2vx+2vz_0+t\eta|)^{\epsilon_0-1}\\ 
  & \qquad \qquad\qquad\qquad\qquad\qquad |\eta| |\chi(\eta)| |\varphi(y)||\varphi(x)| \ dtdy\dbar\eta dx\\
  & \leq  C' (1-\epsilon + |z_0|)^{\epsilon_0 -1}.
\end{align*}
In the last line, we used that on the support of $\varphi(x)$, $|x|\leq \epsilon$, so \[ |x + z_0| \geq |z_0|-|x| \geq |z_0| - \epsilon \] which gives the above inequality. Similarly, when $\eta$ is bounded away from zero, we can integrate by parts, obtaining the bounds (using $v\in [-1,1]$)
 \begin{align*}
 & \Big | \int \int_{t=0}^{t=1} \frac{1}{(-i\eta)^3}\partial_y^3 (e^{i(x-y)\eta}) p_2(x+z_0,2vx+ 2vz_0+t\eta)\\
 & \qquad \qquad\qquad\qquad\qquad \eta(1- \chi(\eta))e^{i(y-x)^2v} \varphi(y)\varphi(x) \ dt dy\dbar\eta dx\Big| \\
 & =  \Big| \int \int_{t=0}^{t=1} \frac{1}{\eta^3}(e^{i(x-y)\eta}) p_2(x+z_0,2vx+ 2vz_0+t\eta)\\
 & \qquad \qquad\qquad\qquad\qquad \eta(1- \chi(\eta)) \partial_y^3 (e^{i(y-x)^2v} \varphi(y))\varphi(x) \ dt dy\dbar\eta dx\Big|\\
  & \leq C''  (1 - \epsilon + |z_0|)^{\epsilon_0-1} \int \int_{t=0}^{t=1} \frac{|1-\chi(\eta)|}{|\eta|^2} |\partial_y^3 (e^{i(y-x)^2v} \varphi(y))| |\varphi(x)| \  dt dy\dbar\eta dx\\
   & \leq C'''  (1 - \epsilon + |z_0|)^{\epsilon_0-1}.
\end{align*}

Combining all of this into \eqref{eq:def-I} we thus have \begin{equation}\label{eq:p-bounds} \int p(x+z_0, 2vx+2vz_0)\varphi(x)^2 dx + R(z_0) \leq C \Vert \varphi \Vert_{L^2(\RR)}^2, \end{equation} where $R(z_0)$ is the sum of the above two terms, and by the above bounds, we know that $R(z_0) \to 0$ as $|z_0|\to\infty$. Because all derivatives of $p$ go to zero, $p$ is uniformly continuous, and thus for $\delta>0$, we can take $\epsilon$ in the definition of $\varphi$ small enough so that \[\int [\sup_{x' \in B_\epsilon(0)} [p(x'+z_0,2vx'+2vz_0)] -p(x+z_0,2vx+2vz_0) ]\varphi(x)^2 \ dx < \delta .\] This combined with \eqref{eq:p-bounds} gives bounds of the form \[ |p(z_0, 2vz_0)| \leq P\] for some $P\geq0$ (depending on $\epsilon$) and for large $|z_0|$ (and thus for all $z_0$ because $p$ is certainly bounded inside of a compact set). From the above proof, it is clear that $P$ does not depend on $v\in [-1,1]$, so we know that \[ \sup_{|x| \geq  2|\xi|}| p(x,\xi)|\leq P\] by taking appropriate $v$ and $z_0$ in these bounds. For $q$ the symbol of $Q$, this shows that 
\[ \sup_{|x| \geq  2|\xi|}| q(x,\xi)|\leq \sqrt P.\]
Now, to extend this to all of $\RR\times \RR$, taking the Fourier transform of the operator $Q$, denoted $\hat Q$ and applying the above argument to $\hat Q$ (which is still bounded on $L^2$ because the Fourier transform is an isometry $L^2 (\RR)\to L^2(\RR)$) we have that 
\[ \sup_{|x| \geq  2|\xi|}| q(\xi,x)|\leq \sqrt {P'}.\] This shows that $q$ is bounded.
\end{proof}


\bibliography{ChodoshREPPSIDObibliography}
\bibliographystyle{siam}
\end{document}